\documentclass[12pt,psamsfonts]{amsart}

\newtheorem{theorem}{Theorem}[section]
\newtheorem{lemma}[theorem]{Lemma}
\newtheorem {corollary} [theorem]{Corollary}

\theoremstyle{definition}

\theoremstyle{remark}

\numberwithin{equation}{section}

\usepackage{graphicx}
\usepackage{amssymb}
\usepackage{amscd}
\usepackage{amsmath}
\usepackage{amsfonts}
\usepackage{amsbsy}
\usepackage{epsfig}

\begin{document}

\title[Analytic normalization of integrable systems ]{Analytic normalization of analytically integrable differential systems near a periodic orbit}

% author one information
\author[K. Wu]{Kesheng Wu}
\address{Department of Mathematics, Shanghai Jiao Tong University,  Shanghai 200240,  People's Republic of China.}
%\curraddr{}
\email{E-mail: kshengwu@gmail.com}
%\thanks{The author is partially supported by  FP7-PEOPLE-2012-IRSES-316338 of Europe.}

%    author two information
\author[X. Zhang]{Xiang Zhang}
\address{Department of Mathematics, and MOE--LSC, Shanghai Jiao Tong University,  Shanghai 200240,  People's Republic of China.}
%\curraddr{}
\email{xzhang@sjtu.edu.cn}
\thanks{The second author is partially supported by
NNSF of China grant 11271252, and RFDP of Higher Education of China grant 20110073110054. Both authors are also supported by FP7-PEOPLE-2012-IRSES-316338 of Europe.}


\date{}

\dedicatory{}


\begin{abstract}
For an analytic differential system in $\mathbb R^n$ with a periodic orbit, we will prove that if the system is analytically integrable around the periodic orbit, i.e. it has $n-1$ functionally independent analytic first integrals defined in a neighborhood of the periodic orbit, then the system is analytically equivalent to its Poincar\'e--Dulac type normal form.  This result is an extension for analytic integrable differential systems around a singularity to the ones around a periodic orbit.

\end{abstract}

\maketitle

\bigskip
\section{Introduction and statement of the main results}\label{s1}

\setcounter{section}{1}
\setcounter{equation}{0}\setcounter{theorem}{0}
Normal form theory has been playing key roles in the study of dynamics for ordinary differential equations and dynamical systems (smooth and discrete ones). Because of its importance, it has been extensively studied, see for instance \cite{AA,Ar,BT2003,Bib,Bi,Br,CLW,CY,Il79,IY2008,Li,MZ,PM1,PM2,Ta,TZ,WZ,WZ2,Zh,Zu1,Zu} and the references therein. In the normal form theory, one of the main topics is to study the existence of analytic normalization for an analytic dynamical system to its normal form. In this direction there are lots of well known results, which involved the works of Poincar\'{e} \cite{Po,Po1}, Siegel \cite{Ar}, Bruno \cite{Br}, Ilyashenko \cite{IY2008} and so on.

Here we mainly concern the analytically integrable systems.  Along this direction Zung \cite{Zu}  showed via torus action that any analytic vector field which is  analytic integrable in a neighborhood of the origin in the non-Hamiltonian sense admits a convergent Poincar\'{e}--Dulac normalization. Zhang \cite{Zh} presented a similar result using analytic methods and provided the explicit expression of the normal form, which was not obtained in \cite{Zu}. Furthermore Zhang \cite{Zh2013} extended the results in \cite{Zh} to degenerate cases and also obtained a result on analytically integrable diffeomorphisms around a fixed point. On the existence of analytic normalization of analytically integrable differential systems, we refer the readers to \cite{LPW}, too.

About the integrability and normalization of a differential system  near a periodic orbit, Yakovenko \cite{Ya93} studied the existence of $C^\infty$ normalization of a planar $C^\infty$ differential system near a periodic orbit to  a very simpler normal form.  Peralta--Salas \cite{Pe08} presented a characterization between integrability and normalizers of smooth vector fields in a region fulled up with periodic orbits.
But as our knowledge it is unsolved that the problem on whether an analytically integrable differential systems around a periodic orbit is analytically equivalent to its Poincar\'e--Dulac normal form.

\smallskip

Consider the analytic differential system
\begin{equation} \label{F1}
\dot{x} = f(x), \qquad    \qquad   x \in\Omega\subset \mathbb{R}^n,
\end{equation}
where the dot denotes the derivative with respect to the time $t$, $\Omega$ is an open subset of $\mathbb R^n$ and $f(x)\in C^\omega(\Omega)$. Here $C^\omega(\Omega)$ denotes the ring of analytic functions defined in $\Omega$. Assume that system \eqref{F1} has a periodic orbit, saying $\Gamma$, located in the region $\Omega$.

We say that system \eqref{F1} is {\it analytically integrable in a neighborhood of $\Gamma$}, if it has $n-1$ functionally independent analytic first integrals defined in the neighborhood of $\Gamma$. Here we say that $k>1$ first integrals defined in $D\subset \Omega$ are {\it functionally independent} if the gradients of the $k$ first integrals have rank $k$ in a full Lebesgue measure subset of $D$. A nonconstant function $H(x)$ is a {\it first integral} of system \eqref{F1} in $D$ if along any orbit located in $D$ of system \eqref{F1}, the function $H$ takes a constant value.

Let $x = \varphi(t)$ be an expression of $\Gamma$ with period $T$. Since system \eqref{F1} is analytic, the periodic solution $\phi(t)$ is also analytic on $\mathbb R$. Taking the transformation $X = x - \varphi(t)$, system \eqref{F1} becomes
\begin{equation} \label{F2}
\dot{X} = f(X + \varphi(t)) - f(\varphi(t)).
\end{equation}
It can be written in the form
\begin{equation}\label{F3}
\dot{X} = A(t) X + g(X, t),  \qquad  \quad  g(X, t) = O(\left| X \right|^2),
\end{equation}
with $A(t)$ analytic and periodic in $t$ of period $T$, and $g(X, t)$ analytic in $X$ and $t$ and periodic in $t$ of period $T$.

By the Floquet theory \cite{GH},  there is a change of coordinates $X = Q(t)Y$ with $Q(t)$ invertible, analytic and periodic of period $T$, under which  system  \eqref{F3} is transformed to
\begin{equation}\label{F3.1}
 \dot{Y} = A Y + h(Y, t) ,
 \end{equation}
where $A$ is a constant matrix (real or complex), and $h=O(|Y|^2)$ is analytic in its variables and periodic in $t$ of period $T$. Thus we reduce system \eqref{F1} near the periodic orbit $\Gamma$ to system \eqref{F3.1} with the constant solution $Y=0$. It is well known that the constant matrix $A$ has a zero eigenvalue with its characteristic direction tangent to the periodic orbit $\Gamma$. In what follows, for  convenience we still use variable $x$ to replace $Y$ in \eqref{F3.1}.

Let $\lambda=(\lambda_1, ..., \lambda_n )$ be the $n$--tuple of eigenvalues of $A$. For the analytic function $h(x,t)$ periodic in $t$, expanding it in Taylor series in $x$ and Fourier series in $t$ as follows
\[
h(x, t) = \sum^n_{j=1} \sum_{l \in \mathbb Z_+^n} \sum_{k \in \mathbb{Z}} a_{l,k,j}  x^l e^{{\bf i} k t} e_j,
\]
where $\mathbf i=\sqrt{-1}$, $\mathbb Z_+$ is the set of nonnegative integers, $e_j$ is the unit vector with its $j$th component equal to $1$, and $x^l=x_1^{l_1}\ldots x_n^{l_n}$.
We call the pseudomonomial $x^l e^{{\bf i} k t} e_j$ in the $j$th component of $h(x,t)$ {\it resonant} if
\[
{\bf i} k + \langle l,  \lambda\rangle - \lambda_j = 0,
\]
where $\langle\cdot,\cdot\rangle$ denotes the inner product of two $n$--dimensional vectors.

System \eqref{F3.1} is in the {\it Poincar\'e--Dulac normal fom}  if $h(x, t)$ contains resonant pseudomonomials only. We should say that the Poincar\'e--Dulac normal form defined here is an extended version of the classical one for autonomous differential systems.

System \eqref{F3.1} is {\it formally equivalent to its Poincar\'e--Dulac normal form} if there exists a tangent to identity transformation of the form
\begin{equation}\label{F3.2}
x = y + \Phi(y,t),
\end{equation}
with $\Phi(y,t)=O(|y|^2)$ a formal series in $y$ and periodic in $t$, which transforms \eqref{F3.1} to a system
\begin{equation}\label{F3.3}
\dot{y} = A y + G(y, t),
\end{equation}
which is in the Poincar\'e--Dulac normal form. Furthermore,
\begin{itemize}

\item If the transformation \eqref{F3.2} is analytic, we say that system \eqref{F3.1} is {\it analytically equivalent} to its Poincar\'e--Dulac normal form.

\item If $\Phi(y,t)$ in the transformation \eqref{F3.2} contains only non--resonant pseudomonomials, i.e $x^l e^{{\bf i} k t} e_j$ in $\Phi(y,t)$ satisfies $\mathbf i k+\langle l,\lambda\rangle\ne 0$,  we say that system \eqref{F3.1} is analytically equivalent to its {\it distinguished normal form}. The transformation \eqref{F3.2} is called {\it distinguished normalization}.
\end{itemize}
We should mention the difference between the resonances of the pseudomonomials in a vector field and in a transformation (or a function).

Now we can state our main results.

\begin{theorem}\label{th1}
Assume that system \eqref{F1} is analytic and has a periodic orbit. If system \eqref{F1} is analytically integrable in a neighborhood of the periodic orbit,  then the system is analytically equivalent to its distinguished normal form in a neighborhood of the periodic orbit.
\end{theorem}

Theorem \ref{th1} will be proved in the next section.

We should mention that Theorem \ref{th1} is an extension of the results given in \cite{Po,Po1,LPW,Zh,Zh2013,Zu} for local analytically integrable differential systems around a singularity to the ones around a periodic orbit. As shown in \cite{Zh,Zh2013} for analytic integrable differential systems \eqref{F1} around a singularity $S$ we must have that the Jordan normal form of $\partial_x f(S)$ is diagonal, where $\partial_x f(S)$ denotes the Jacobian matrix of $f$ at $S$.
Here in general we do not know whether the Jordan normal form of $A$ in \eqref{F3.1} is diagonal. And so it increases the difficulty in the proof of the convergence of the normalization from system \eqref{F3.1} to its distinguished normal form.

Next result exhibits the properties of the characteristic exponents of periodic orbits of analytically integrable differential system.

\begin{corollary}\label{la4}
If the analytic differential system \eqref{F1} has $n - 1$ functionally independent analytic or formal first integrals around a periodic orbit, then the characteristic exponents of the periodic orbit satisfy
\[
\lambda_1=\nu_1 \mathbf i,\ldots, \lambda_{n-1}=\nu_{n-1} \mathbf i, \lambda_n=\nu_n\mathbf i, \quad \nu_j\in \mathbb Q.
\]
\end{corollary}

The proof of Corollary \ref{la4} can be obtained as an easy consequence of the proof of Lemma \ref{la3}. The details will be omitted.

This last result presents a new phenomena on the characteristic exponents of a periodic orbit for an analytically integrable differential system.
Recall that a {\it characteristic multiplier} of the periodic orbit $\Gamma$ of system \eqref{F1} is by definition the eigenvalues of the monodromy operator of the linear part, i.e. $\dot x=A(t)x$, of system \eqref{F3}. A number $\mu$ is a {\it characteristic exponent} of the periodic orbit $\Gamma$ if $e^{\mu T}$ is a characteristic multiplier, where $T$ is the period of $\Gamma$. We must say that if $\mu$ is a characteristic exponent, then $\mu+2m\pi \mathbf i/T$, $m\in\mathbb Z$, are also  characteristic exponents. Here we do not consider these later exponents. In fact, here the characteristic exponents are the eigenvalues of the constant matrix $A$ in \eqref{F4}.

The remaining part is the proof of our main results.

\section{Proof of Theorem \ref{th1}}\label{s3}

In order to prove Theorem \ref{th1}, we first prove the existence of distinguished formal normal form of  system \eqref{F1}.
\smallskip

\subsection{Distinguished formal normal form}\label{s3.1}

As we described in Section 1,  system \eqref{F1} near the periodic orbit $\Gamma:\, x = \varphi(t)$ with period $T$ can be reduced by an invertible analytic change of variables to the normal form system
\begin{equation}\label{F4}
\dot{x} = A x + F(x,t), \qquad \qquad x \in \mathbb{R}^n, \quad t \in \mathbb{R}
\end{equation}
where the constant matrix $A$ has a zero eigenvalue and $F(x,t)=O(|x|^2)$ is analytic and periodic in $t$.

\begin{lemma}\label{lnf}
The analytic periodic differential system \eqref{F4} is formally equivalent to its distinguished normal form.
\end{lemma}

\begin{proof}
Assume that system \eqref{F4} is transformed to
\begin{equation}\label{FF1}
\dot{y} = A y + G(y,t)
\end{equation}
via the transformation
\begin{equation}\label{FF2}
x = y + \Phi(y,t),
\end{equation}
where $G(y,t)$ and $\Phi(y,t)$ are formal series in $y$ without constant and linear terms and periodic in $t$ of period $T$. Then $\Phi(y,t)$ and $G(y,t)$ satisfy the equations
\begin{equation}\label{FF2.1}
\partial_t \Phi + \langle\partial_y \Phi, A y\rangle - A \Phi = F(y + \Phi, t) - \langle\partial_y \Phi , G\rangle - G,
\end{equation}
where $\partial_t\Phi$ denotes the partial derivative of $\Phi$ with respect to $t$, and $\partial_y\Phi$ denotes the Jacobian matrix of $\Phi$ with respect to $y$.

Expanding $F,G,\Phi$ in Taylor series in $y$
\[
H(y,t)=\sum\limits_{j=2}\limits^\infty H_j(y,t), \qquad H\in\{F,G,\Phi\},
\]
where $H_j$ are homogeneous polynomials in $y$ of degree $j$ with coefficients periodic functions in $t$ of period $T$.
Substituting these expansions in \eqref{FF2.1} and equating the homogeneous terms in $y$ with the same degree gives
\begin{equation}\label{F5}
\partial_t \Phi_s + \langle\partial_y \Phi_s, A y\rangle - A \Phi_s = [F]_s - \sum^{s-1}_{j=2}\partial_y \Phi_j G_{s+1-j} - G_{s},
\end{equation}
where $[F]_s$ are inductively known vector--valued homogeneous polynomial in $y$ of degree $s$ obtained after  re--expanding $F\left(y + \sum\limits_{j=2}\limits^{s-1}\Phi_j(y,t), t\right)$ in the power series of $y$.

Set
\begin{equation}\label{F5.0}
W_s(y,t)=[F]_s - \sum^{s-1}_{j=2}\partial_y \Phi_j G_{s+1-j}.
\end{equation}
Clearly $W_s$ are inductively known. Expanding the coefficients of $V_s$, $V\in\{\Phi, G, W \}$, in the Fourier series of $t$, we get
\[
V_s(y,t)=\sum\limits_{k\in\mathbb Z}V_s^k(y)e^{\mathbf i kt}.
\]

For studying the existence of solutions of \eqref{F5}, we need the following result, which is due to Bibikov \cite{Bib}.
\begin{lemma}\label{la1}
Denote by $\mathcal{G}^r(\mathbb{F})$, $\mathbb F=\mathbb R$ or $\mathbb C$, the linear space of n-dimensional vector--valued homogeneous
polynomials of degree $r$ in $n$ variables with coefficients in $\mathbb{F}$. Let $A$ and $B$ be two
square matrices with entries in $\mathbb{F}$, and their n--tuples of eigenvalues be $\lambda$ and $\kappa$,
respectively. Defining an operator $L^*$ on $\mathcal{G}^r(\mathbb{F})$ by
\[
L^* h = \langle\partial_x h, A x\rangle - B h,  \qquad \qquad h \in \mathcal{G}^r(\mathbb{F}).
\]
Then the spectrum of the operator $L^*$ is
\[
\sigma(L^*) := \{\langle l,\lambda\rangle - \kappa_j; \,\, l \in \mathbb{Z}_+^n, |l|=r, j=1,...,n\}.
\]
\end{lemma}

Set $\mathcal{G}^s(\mathbb{F}, t)$  be the $\mathbb F$--linear space  spanned by the base
\[
\left\{x^l e^{{\bf i} k t} e_j; \,\,\,\, l \in \mathbb{Z}_+^n, \,\, |l|=s, \,\,k\in \mathbb Z, \,\, j=1,\ldots,n \right\}.
\]
Defining a linear operator on $\mathcal G^s(\mathbb F,t)$ by
\[
L := \partial_t + \langle\partial_y, A y\rangle - A.
\]
Then it follows from Lemma \ref{la1} that the spectrum of $L$ in  $\mathcal{G}^s(\mathbb{F}, t)$ is
\[
\{ {\bf i} k + \langle l, \lambda\rangle  - \lambda_j;\,\, k \in \mathbb{Z}, l \in \mathbb{Z}_+^n, |l| = s,  j=1,...,n\}.
\]

Separate $\mathcal{G}^s(\mathbb{F}, t) = \mathcal{G}_0^s(\mathbb{F}, t) \oplus \mathcal{G}_1^s(\mathbb{F}, t)$
in such a way that $\mathcal G_1^s(\mathbb F,t)$ is formed by non--resonant pseudomonomials and $\mathcal G_0^s(\mathbb F,t)$ is formed by resonant pseudomonomials.  Clearly $L$ is invertible on $\mathcal G_1^s(\mathbb F,t)$ and $L(\mathcal G_0^s(\mathbb F,t))\subset\mathcal G_0^s(\mathbb F,t)$.
 Of course, if $A$ is diagonal then $L|_{\mathcal G_0^s(\mathbb F,t)}=0$.

According to the decomposition of $\mathcal G^s(\mathbb F,t)$, we separate the right--hand side of \eqref{F5} in two parts with one component resonant and another one nonresonant. Then equation \eqref{F5}  can be written in two equations
\begin{eqnarray}\label{F5.1}
L \Phi_{s,r}   & = &  W_{s,r} - G_{s,r},\\
L \Phi_{s,nr} &=& W_{s,nr} - G_{s,nr},
\label{F5.2}
\end{eqnarray}
where $r$ and $nr$ in the subscription denote the resonant and nonresonant terms, and $V_s=V_{s,nr}+V_{s,r}$, $V\in\{\Phi, G , W\}$ with $W_s$ defined in \eqref{F5.0},

For equation \eqref{F5.1} we choose
\[
G_{s,r} = W_{s,r}.
\]
Then the equation has always the solution $\Phi_{s,r} = 0$.

For equation \eqref{F5.2} the operator $L$ is invertible in $\mathcal G_1^s(\mathbb F,t)$, so for any $G_{s,nr}\in\mathcal G_1^s(\mathbb F,t)$ equation \eqref{F5.2} has a unique solution in $\mathcal G_1^s(\mathbb F,t)$. We set $G_{s,nr}=0$, equation \eqref{F5.2} has the unique solution
\[
\Phi_{s,nr} = L^{-1}(W_{s,nr}).
\]
Clearly $\Phi_{s,nr}\in \mathcal G_1^s(\mathbb F,t)$.

From the above construction, we get that
\[
\Phi(y,t)=\sum\limits_{s=2}\limits^\infty \Phi_{s,nr}(y,t),
\]
consists of only nonresonant pseudomonomials. So the formal normalization $x = y + \Phi(y,t)$ is distinguished, and the distinguished formal normal form is of the form
\[
\dot{y} = A y + \sum^\infty_{s = 2} G_{s,r}(y, t).
\]
This proves that system \eqref{F4} is formally equivalent to its distinguished normal form. We complete the proof of the lemma.
\end{proof}

Next we will study the relation between the first integrals of the original system \eqref{F4} and of the distinguished normal form system \eqref{FF1}. Furthermore we will present some properties on the resonance of eigenvalues of $A$ when system \eqref{F1} is analytically integrable in a neighborhood of the periodic orbit $\Gamma$.

\subsection{First integrals and resonances}\label{s3.2}
In this subsection, we first study the structure of first integrals for the distinguished normal form systems of system \eqref{F1} with first integrals.

Recall that a pseudomonomial $y^l e^{{\bf i} k t}$ in a (vector--valued) function is {\it resonant} if
 \[
 {\bf i} k + \langle \lambda, l\rangle = 0,
 \]
where $\lambda = (\lambda_1, ..., \lambda_n )$ are the $n$--tuple of eigenvalues of the matrix $A$. A periodic function   $W(y,t)$ is {\it resonant} if its expansion
\[
W(y,t) = \sum_{l \in \mathbb Z_+^n} \sum_{k \in \mathbb{Z}} a_{l,k}  x^l e^{{\bf i} k t},
\]
has all its pseudomonomials resonant.

The next result characterizes the first integrals of a distinguished normal form system of system \eqref{F4} with first integrals around the trivial solution $x=0$.

\begin{lemma}\label{la2}
Assume that system \eqref{F4} has an analytic or a formal first integral $H(x,t)$ which  is periodic in $t$ of period $T$ and that system \eqref{FF1} is the distinguished normal form of \eqref{F4} under the
 normalization \eqref{FF2}. The following statements hold.
 \begin{itemize}
 \item[$(a)$] $\tilde{H}(y,t) := H(y + \Phi(y, t), t)$ is a first integral of \eqref{FF1}.
 \item[$(b)$] $\tilde H(y,t)$ is resonant.
 \end{itemize}
\end{lemma}

\begin{proof} $(a)$
From the assumption, we know that the first integral $H(x,t)$ of system \eqref{F4} satisfies
\[
\partial_t H(x,t) + \langle \partial_x H(x,t), A x + F(x,t)\rangle = 0.
\]
Substituting $x=y+\Phi(y,t)$ to this last equality, we get
\begin{eqnarray*}
&& \partial_t H(y + \Phi(y,t),t) + \langle \partial_x H(y + \Phi(y,t),t),A y\rangle \\
&& \qquad =-
\langle \partial_x H(y + \Phi(y,t),t), A \Phi(y,t) + F(y + \Phi(y,t),t)\rangle.
\end{eqnarray*}
On the other hand, some direct calculations show that
\begin{eqnarray*}
\partial_t \tilde{H}(y,t) &=& \partial_t H(y + \Phi(y,t),t) + \partial_x H(y + \Phi(y,t), t) \partial_t \Phi(y,t), \nonumber \\
\partial_y \tilde{H}(y,t) &=& \partial_x H(y + \Phi(y,t),t) ( E+ \partial_y \Phi),  \nonumber
\end{eqnarray*}
where $E$ is the $n$th order unit matrix.

Using these equalities, we get that
\begin{eqnarray}
\partial_t \tilde{H}(y,t) &+ &\langle\partial_y \tilde{H}(y,t), Ay + G\rangle \nonumber\\
&=& \partial_t H(y + \Phi, t) + \partial_x H(y + \Phi, t)\partial_t \Phi \nonumber \\
                                                                    & & + \langle \partial_x H(y+\Phi, t)(1 + \partial_y \Phi), Ay + G\rangle    \nonumber \\
                                                               &=& \left\langle\partial_x H(y+\Phi,t),\, - A \Phi -F(y + \Phi, t)\right.  \\
                                                               & & \qquad\qquad\qquad\quad \left.  +\partial_t \Phi + G +   \partial_y \Phi (Ay+G)\right\rangle \equiv 0.    \nonumber
\end{eqnarray}
where in the last equality we have used \eqref{FF2.1}, i.e. the transformation $\Phi$ from \eqref{F4} to its distinguished normal form \eqref{FF1} satisfies the equality
\[
\partial_t \Phi + \langle\partial_y \Phi, A y\rangle - A \Phi = F(y + \Phi, t) - \langle\partial_y \Phi , G\rangle - G.
\]
This proves that $\tilde{H}(y,t)$ is a first integral of system \eqref{FF1}. So the statement holds.

\smallskip

\noindent $(b)$
Write
\begin{equation}\label{FFel}
\tilde{H}(y,t) = \sum^{\infty}_{k=l} \tilde{H}_k (y,t),
\end{equation}
with $l \in \mathbb N$ and $ \tilde{H}_k(y,t) \in \mathcal{G}^k (\mathbb{F},t)$, $k=l,l+1,\ldots$ Since $\tilde{H}(y,t)$ is a first integral
of system \eqref{FF1} by statement $(a)$, it follows that
\begin{equation}\label{FFfl}
\partial_t\tilde{H}_l + \langle\partial_y \tilde{H}_l, Ay\rangle = 0.
\end{equation}
 We define a new linear operator on $\mathcal{G}^k (\mathbb{F},t)$
\[
\tilde{L} := \partial_t + \langle \partial_y, A y\rangle.
\]
It follows from Lemma \ref{la1} that the spectrum of $\tilde{L}$ is
\[
\{ {\bf i} k + \langle\lambda, s\rangle;\,\,  s \in \mathbb{Z}_+^n, |s|=l, k \in \mathbb{\mathbb{Z}}\}.
\]
Working in a similar way as that in the last subsection we can prove that equation \eqref{FFfl} has only solutions $\tilde{H}_l$ which consist of resonant pseudomonomials.

Now we will use the induction to prove  that all $\tilde{H}_k $, $k=l+1,l+2,\ldots$, are resonant. By induction we assume that for any given $m > l$, $\tilde{H}_j$, $j=l, ..., m-1$, are all resonant.

By the expansion \eqref{FFel} of the first integral $\tilde H(y,t)$, some easy computations show that
\begin{equation}\label{FFfl1}
\tilde{L}(\tilde{H}_m(y,t)) + \sum^{m}_{j=2} \left\langle\partial_y \tilde{H}_{m+1-j}(y,t), G_j(y,t) \right\rangle=0,
\end{equation}
where $\tilde H_s=0$ for $s<l$.
We note that $G_j$ and $\tilde{H}_{m+1-j}$ are resonant homogeneous polynomials in $y$ in a vector field and in a function, respectively. Then some further calculations show that all the terms in the last summation of \eqref{FFfl1} are resonant as a function. So we get from the spectrum of $\tilde{L}$ that the solution $\tilde{H}_m (y,t)$ of equation \eqref{FFfl1} consists of resonant pseudomonomials only. By induction we have proved statement $(b)$ and consequently the lemma.
\end{proof}

Next we will study the properties of the nonresonant eigenvalues of $A$ for analytically integrable system \eqref{F4}, which is a key point in the proof of our main result.
\begin{lemma}\label{la3}
If system \eqref{FF1} has $n - 1$ functionally independent analytic or formal first integrals, then there exists $\epsilon > 0$ such that for all ${\bf i}k + \langle m, \lambda\rangle - \lambda_i \neq 0,
k \in \mathbb{Z}, m \in \mathbb{Z}_+^n, |m| \geq 2$, we have
\[
\left|{\bf i}k + \langle m, \lambda\rangle - \lambda_i\right| > \epsilon.
\]
\end{lemma}

\begin{proof}
Recall that the constant matrix $A$ in \eqref{FF1} has a zero eigenvalue, which was prescribed in Section \ref{s1}. Without loss of generality we set $\lambda_n = 0$.
By Theorem 1.1 of  \cite{CY}, we know that the number of functionally independent analytic or formal first integrals of system \eqref{FF1} is less than or equal to $R_\lambda$, where $R_\lambda$ is the rank of the $\mathbb Z$--linear space spanned by
\[
\left\{(k,l); \,\, {\bf i}k + \langle \lambda, l\rangle = 0, k \in \mathbb{Z}, l \in \mathbb{Z}_+^n\right\},
\]
Recall that $\lambda = (\lambda_1, ...,\lambda_n)$
is the $n$--tuple of eigenvalues  of $A$. By Lemma \ref{la2} and  the assumption of this lemma we get that $R_\lambda \geq n-1$.  On the other hand, we have clearly $R_\lambda \leq n$.

\noindent{\it Case }1. $R_\lambda = n-1$. Then there are $n - 1$ linearly independent vectors in $\mathbb{Z} \times \mathbb{Z}_+^n$, saying
\[
(p_1, m_{1,1}, ..., m_{1,n}),\,\, (p_2, m_{2,1}, \,\, ...,\,\,  m_{2,n}), \,\, ..., \,\, (p_{n-1},\,\, m_{n-1,1},\,\, ..., \,\,m_{n-1,n}),
\]
such that
\begin{eqnarray}
{\bf i} p_1 + m_{1,1} \lambda_1 + ... + m_{1,n-1}\lambda_{n-1} = 0,   \nonumber \\
{\bf i} p_2 + m_{2,1} \lambda_1 + ... + m_{2,n-1}\lambda_{n-1} = 0,   \nonumber \\
\vdots\qquad\qquad\qquad\qquad      \nonumber \\
{\bf i} p_{n-1} + m_{n-1,1} \lambda_1 + ... + m_{n-1,n-1}\lambda_{n-1} = 0.   \nonumber
\end{eqnarray}

Without loss of generality, we  just need to consider the following two cases:
\begin{equation}\label{F7.1}
\mbox{det} \left (\begin{array}{cccc}
m_{1,1} &  m_{1,2} & ... & m_{1,n-1}\\
m_{2,1} &  m_{2,2} & ... & m_{2,n-1}\\
... &  ... & ... & ...\\
m_{n-1,1} &  m_{n-1,2} & ... & m_{n-1,n-1}\\
\end{array}\right )\neq 0,
\end{equation}
or
\begin{equation}\label{F7.2}
\mbox{det} \left (\begin{array}{cccc}
p_{1} &  m_{1,2} & ... & m_{1,n-2}\\
p_{2} &  m_{2,2} & ... & m_{2,n-2}\\
... &  ... & ... & ...\\
p_{n-1} &  m_{n-1,2} & ... & m_{n-1,n-2}\\
\end{array}\right )\neq 0.
\end{equation}
In Case \eqref{F7.1}, we have
\[
\left( \begin{array}{c}
\lambda_1 \\
\vdots   \\
\lambda_{n-1}
\end{array}
\right) = \left( \begin{array}{ccc}
m_{1,1} & ...  &m_{1,n-1}  \\
\vdots & \vdots  & \vdots  \\
m_{n-1,1} & ...  &m_{n-1,n-1}
\end{array}
\right )^{-1}  \left( \begin{array}{c}
- p_1 {\bf i} \\
\vdots \\
- p_{n-1} {\bf i}
\end{array}
\right).
\]
Thus, it follows that
\[
\lambda_1 = \frac{\mu_1}{\nu} {\bf i}, \lambda_2= \frac{\mu_2}{\nu} {\bf i}, ..., \lambda_{n-1} = \frac{\mu_{n-1}}{\nu} {\bf i},
\]
with $\mu_i \in \mathbb{Z}$ for $i = 1, 2, ..., n-1$, and $\nu \in \mathbb N$.

For ${\bf i} k + \langle m, \lambda\rangle - \lambda_j \neq 0$, $k\in \mathbb Z$, $ m \in \mathbb Z_+^n$ and $|m| \geq 2$, the following holds
\[
\left| {\bf i}k +  \langle m, \lambda\rangle - \lambda_j \right| = \left| \frac{k\nu + m_1 \mu_1 + ... + m_{n-1} \mu_{n-1}-\mu_{j}}{\nu}\right| \geq \frac{1}{\nu}.
\]

In Case \eqref{F7.2}, we have
\[
\left( \begin{array}{c}
{\bf i}\\
\lambda_1 \\
\vdots \\
\lambda_{n-2}
\end{array}
\right) = \left( \begin{array}{cccc}
p_1 & m_{1,1} & ...  &m_{1,n-2}  \\
p_2 & m_{2,1} & ... & m_{2,n-2}\\
\vdots & \vdots  &  \vdots & \vdots  \\
p_{n-1} & m_{n-1,1} & ...  &m_{n-1,n-2}
\end{array}
\right )^{-1} \left( \begin{array}{c}
m_{1,n-1} \lambda_{n-1} \\
m_{2,n-1} \lambda_{n-1} \\
\vdots \\
m_{n-1,n-1} \lambda_{n-1}
\end{array}
\right).
\]
Clearly we have $\lambda_{n-1}\ne 0$.
Then for ${\bf i} k + \langle m, \lambda\rangle - \lambda_j \neq 0$, $k\in\mathbb Z$, $ m \in \mathbb Z_+^n$ and $|m| \geq 2$, we also have
\[
\left| {\bf i}k +  \langle m, \lambda\rangle - \lambda_j \right|  \geq \frac{1}{\mu},
\]
for some $\mu\in\mathbb N$.

\noindent{\it Case }2. $R_\lambda = n$. Recall that $\lambda_n = 0$.  There are $n$ linearly independent vectors in $\mathbb{Z} \times \mathbb{Z}_+^n$, saying
\begin{equation}\label{F7.3}
(0, 0, \, ...,\, 1),\, (p_1, m_{1,1},\,  ...,\,  m_{1,n}), \,  ..., \, (p_{n-1}, m_{n-1,1}, \, ...,\,  m_{n-1,n}),
\end{equation}
such that
\[
\left(
\begin{array}{cccc}
0 & 0 & ... & 1\\
p_1 & m_{1,1} &  ... &  m_{1,n}\\
\vdots & \vdots & \vdots & \vdots \\
p_{n-1} &  m_{n-1,1} & ... &  m_{n-1,n}
\end{array}
\right) \left(\begin{array}{c}
\mathbf i\\
\lambda_1\\
\vdots\\
\lambda_n
\end{array}
\right)=0.
\]
From \eqref{F7.3} we get that the vectors
\[
(p_1, m_{1,1}, ..., m_{1,n-1}),  ..., (p_{n-1}, m_{n-1,1}, ..., m_{n-1,n-1}) \in \mathbb{Z} \times \mathbb{Z}_+^{n-1},
\]
are linearly independent. This implies that we can reduce this case to the case $R_\lambda = n-1$. So, working in a similar way as in the proof of the case $R_\lambda = n-1$, we can finish the proof of the lemma.
\end{proof}

\subsection{Proof of Theorem \ref{th1}}\label{s3.3}
Recall that system \eqref{F1} near a periodic orbit can be reduced to the  periodic normal form  system \eqref{F4} near the trivial solution $x=0$ by an analytic invertible transformation.

In what follows we shall prove the convergence of the normalization which transforms system \eqref{F4} to its distinguish normal form \eqref{FF1}.
By the assumption of the theorem and Lemma \ref{la2} together with the special form of the transformation \eqref{FF2} we get that system \eqref{FF1} has $n-1$ functionally independent formal first integrals.

For $w(z) \in \{f_j, \phi_j,  g_j\}$ with $f_j,\, \phi_j$ and $\phi_j$ the $j$th components of $F, \Phi$ and $G$ respectively, we expand them in the Taylor series in $z$
\[
w(z) = \sum_{k \in \mathbb{Z}_+} \sum_{l \in \mathbb{Z}_+^n,\, |l| = k } w^l (t) z^l,
\]
where the coefficients $w^l(t)$ are periodic functions.

For convenience to notations, we define an order in $\mathbb Z_+^n$: for $k=(k_1, k_2, ..., k_n), l= (l_1, l_2, ..., l_n)\in \mathbb Z_+^n$, we say $k \succ l$, if either
$|k| < |l|$ or  $|k| = |l|$ and there is a number $s \in \{1,2,...,n\}$ such that $k_j = l_j$ for $j \in \{1,2,...,s-1\}$ and $k_s < l_s$.

Using the calculations in \cite{Bib,Li} (see also \cite{Zh,Zh2013}) and the order defined in the last paragraph, we get from equation \eqref{F5}  that $\phi^l_s(t)$ satisfy the equality
\begin{eqnarray} \label{F6}
\partial_t \phi^l_s(t) &+& (\langle l,\lambda\rangle - \lambda_s)\phi^l_s(t) \nonumber\\
&+& \sum^{n}_{j=2} (1 + l_j) \sigma_j \phi^{l-e_{j-1} + e_j}_s(t)- \sigma_s \phi^l_{s-1}(t) \\
&=&\left[ f_{s} (y + \Phi(y,t))\right]^{l} -  g^l_s(t) - \sum_{j=1}^{n} \sum_{k < l, k \in \mathbb{Z}_+^n} \phi^k_s(t) k_j g^{l-k+e_j}_j(t),  \nonumber
\end{eqnarray}
where $\left[ f_{s} (y + \Phi(y,t))\right]^{l}$ is the coefficient of $y^l$ obtained after  re--expanding $f_s(y + \Phi(y,t))$ in the
power series in $y$, and $k<l$ means $k - l \in \mathbb{Z}_+^n$, and
$\sigma_j$ are the elements just under the diagonal entries of the matrix $A$, i.e.
\[
A = \left(\begin{array}{llll}
\lambda_1 & 0 & ... & 0 \\
\sigma_2 & \lambda_2 & ... & 0 \\
\vdots & \ddots & \ddots & \vdots      \\
0& ...& \sigma_n & \lambda_n
\end{array}
\right).
\]
For the Jordan normal form $A$ to be lower triangular, we need the complex coordinates instead of the real ones as did in \cite{Zh2011}.

Further, we expand the periodic  functions $\omega^l_s(t)$, $\omega\in\{\phi,g\}$,  into Fourier series
\[
\omega^l_s(t) = \sum_{m \in \mathbb{Z}} \omega^l_{s,m} e^{{\bf i} m t},
\]
Substituting these expansions in \eqref{F6} and equating the coefficients of $e^{\mathbf i mt}$ give
\begin{eqnarray}\label{F7}
({\bf i} m &+& \langle l, \lambda\rangle - \lambda_s) \phi^{l}_{s, m} \nonumber\\
&+& \sum^{n}_{j=2} (1+l_j) \sigma_j \phi^{l-e_{j-1}+e_{j}}_{s} - \sigma_s \phi^l_{s-1, m}  \\
& =& \left[ f_s(y + \Phi(y, t))\right]^l_{,m} - g^l_{s, m} - \sum^{n}_{j=1} \sum_{k<l} \sum_{m_1+m_2=m} \phi^k_{s,m_1} k_j g^{l-k+e_j}_{j,m_2},\nonumber
\end{eqnarray}
where $\left[ f_s(y + \Phi(y, t))\right]^l_{,m}$, $m\in \mathbb Z$, are the Fourier coefficients in the Fourier expansion of  $\left[ f_s(y + \Phi(y, t))\right]^l$.

Since we consider the distinguish normal form, by Subsection \ref{s3.1}, if
\[
{\bf i} m + \langle l, \lambda\rangle - \lambda_s = 0,
\]
equation \eqref{F7} has the solutions
\begin{eqnarray}\label{F8}
\phi^{l}_{s, m} &=& 0, \nonumber \\
 g^l_{s, m}&=&\sum^{n}_{j=2} (1+l_j) \sigma_j \phi^{l-e_{j-1}+e_{j}}_{s} - \sigma_s \phi^l_{s-1, m} \nonumber\\
& & - \left[ f_s(y + \Phi(y, t))\right]^l_{,m}+  \sum^{n}_{j=1} \sum_{k<l} \sum_{m_1+m_2=m} \phi^k_{s,m_1} k_j g^{l-k+e_j}_{j,m_2}.
\end{eqnarray}
If
\[
{\bf i} m  + \langle l, \lambda\rangle - \lambda_s \neq 0,
\]
equation \eqref{F7} with the choice $g^l_{s, m} = 0$ has the solution
\begin{eqnarray}\label{F9}
\phi^l_{s, m}&=& ({\bf i} m  + \langle l, \lambda\rangle - \lambda_s)^{-1}\nonumber\\
 &&\times \left(  -\sum^{n}_{j=2} (1+l_j) \sigma_j \phi^{l-e_{j-1}+e_{j}}_{s} + \sigma_s \phi^l_{s-1, m}\right.   \\
 & &\left. + \left[ f_s(y + \Phi(y, t))\right]^l_{,m} - \sum^{n}_{j=1} \sum_{k<l} \sum_{m_1+m_2=m} \phi^k_{s,m_1} k_j g^{l-k+e_j}_{j,m_2}\right).  \nonumber
\end{eqnarray}
Summarizing the above calculations, we achieve the distinguished normalization
\[
x_s = y_s + \sum_{l \in \mathbb{Z}_+^n ,|l| > 1, m \in \mathbb{Z}} \phi^l_{s, m} y^l e^{{\bf i} m t},\qquad s=1,\ldots,n,
\]
with $\phi^l_{s, m}$ satisfying \eqref{F9}, and the distinguished normal form
\[
\dot y_s  = (\sigma_s y_{s-1} + \lambda_s y_s) + \sum_{l \in \mathbb{Z}_+^n ,|l| > 1, m \in \mathbb{Z}} g^l_{s, m} y^l e^{{\bf i} m t},\quad s = 1, 2, \ldots, n,
\]
with $g^l_{s,m}$ satisfying \eqref{F8}.

From Lemma \ref{la3}, there exists a positive number $\varepsilon > 0$ such that if
\[
{\bf i} m + \langle l, \lambda\rangle - \lambda_s \neq 0,\,\,  m \in \mathbb{Z},\,
l \in \mathbb{Z}_+^n,\, |l|\geq 2,
\]
we have
\[
\left|{\bf i} m  + \langle l, \lambda\rangle - \lambda_s\right| > \varepsilon.
\]

Next we  estimate $\phi^l_{s,m}$ in \eqref{F9}. Some calculations show that
\begin{eqnarray}\label{F10}
|\phi^l_{s,m}| &\leq& \varepsilon^{-1}\left |\left[ f_s(y + \Phi(y, t))\right]^l_{,m}\right| +  \varepsilon^{-1} \sigma \left |\phi^l_{s-1, m}\right|  \nonumber\\
                & &  + \left|{\bf i} m  + \langle l, \lambda\rangle - \lambda_s\right|^{-1} \sum^{n}_{j=2} (1+l_j) \sigma_j \left|\phi^{l-e_{j-1}+e_{j}}_{s}\right| \nonumber \\
                & & + \left|{\bf i} m  + \langle l, \lambda\rangle - \lambda_s\right|^{-1} \sum^{n}_{j=1} \sum_{k<l} \sum_{m_1+m_2=m} \left |\phi^k_{s,m_1} k_j g^{l-k+e_j}_{j,m_2}\right | \\
                &  := & I + II + III + IV.\nonumber
\end{eqnarray}
From the proof of Lemma \ref{la3} or Lemma \ref{la4}, we know that
 \[
 \lambda_1 = \nu_1 {\bf i},\,\,  ..., \,  \lambda_{n-1} = \nu_{n-1} {\bf i},\,\, \lambda_n =0,\,\,\,\, \nu_j \in \mathbb{Q}
  \]
So there is a positive number $d_1$ such that
\[
|{\bf i} m  +\langle l, \lambda\rangle - \lambda_s|^{-1} (1+l_j) \leq d_1.
\]
This shows that
\[
III \leq d_1 \sigma \sum^{n}_{j=2} |\phi^{l-e_{j-1}+e_{j}}_{s}|,
\]
where $\sigma = \max\{\sigma_2, \sigma_3, ..., \sigma_n\}$.

If $g^{l-k+e_j}_{j,m_2} \neq 0$, then $\langle l-k+e_j, \lambda\rangle = \lambda_j$ and $\langle l, \lambda\rangle =\langle k, \lambda\rangle$. Similarly, there is a positive number $d_2$ such
that
\[
|{\bf i} m  + \langle l, \lambda\rangle - \lambda_s|^{-1} |k_j| =  |{\bf i} m  + \langle k, \lambda\rangle - \lambda_s|^{-1} |k_j| \leq d_2.
\]
Thus we have
\[
IV \leq d_2 \sum^{n}_{j=1} \sum_{k<l} \sum_{m_1+m_2=m} |\phi^k_{s,m_1}| |g^{l-k+e_j}_{j,m_2}|.
\]

Since the function $F(x,t) = (f_1, f_2, ..., f_n)$ is analytic with $x$  in a neighborhood of $x=0$ and $t\in\mathbb R$, by the Cauchy inequality and \cite[Lemma 2.2]{WZ},
there is a set $\mathcal{D} := \{(x,t);\, t\in [0,T], |x_s|\le r,  s=1,2,...,n \}$ such that
\[
|[f_s]^k_{,m}| \leq M r^{-|k|-|m|},   \qquad    M=\max_s \sup_{\partial \mathcal{D}} \{|f_s| \},
\]
where $|k| = k_1 + k_2 + ... + k_n$ and $|m|$ represents the absolute value of $m\in\mathbb Z$.

Define
\[
\widehat{f}(x,t) = M \sum^\infty_{\begin{subarray}{c}  |k| = 2 \\  k\in \mathbb Z_+^n\end{subarray}} \sum_{m \in \mathbb{Z}} r^{-|k|-|m|} x^k e^{{\bf i} m t}.
\]
This is an analytic function in the interior of $\mathcal{D}$, and is a majorant series of $f_s$, $s=1,2,...,n$. In the following, we denote
by $\widehat{w}$ the majorant series of a given series $w$.

Using the estimations on $III$ and $IV$, we get from \eqref{F10} that
\begin{eqnarray}
|\phi^l_{s,m}| &\leq& \varepsilon^{-1} \left |\left[ \widehat{f}(y + \widehat{\Phi})\right]^l_{,m}\right| +  \varepsilon^{-1} \sigma\left |\phi^l_{s-1, m}\right |  \nonumber\\
                & &  + d \sigma \sum^{n}_{j=2} \left |\phi^{l-e_{j-1}+e_{j}}_{s}\right | + d \sum^n_{j=1} \sum_{m_1+m_2=m}\left |(\widehat{\phi}_{s,m_1} \widehat{g}_{j,m_2})^{l + e_j}\right |, \label{F11}
               \end{eqnarray}
where $d = \max\{d_1, d_2\}$.

Since the coefficients in $\widehat{\phi}_s$ are all positive, the convergence of the series $\sum\limits^n\limits_{s=1} \widehat{\phi}_s$ is equivalent to that in the case
$y_1=y_2=...=y_n=z$. Set
\[
\overline{\phi} = \sum^{\infty}_{p=2} \sum_{m \in \mathbb{Z}} \bar{\phi}^p_{,m} z^p e^{{\bf i} m t}, \qquad \bar{\phi}^p_{,m}=\sum_{\begin{subarray}{c} |l|=p\\ l\in \mathbb Z_+^n\end{subarray}}
\sum^n_{s=1} |\phi^l_{s,m}|
\]
Under these notations we get from \eqref{F11} that
\begin{eqnarray*}
\bar{\phi}^p_{,m} &\leq & \varepsilon^{-1} \left|\left[ \widehat{f}(z + \bar{\Phi})\right]^p_{,m}\right | + \varepsilon^{-1} \sigma \bar{\phi}^p_{,m}\\
&&
+ d \sigma (n-1) \bar{\phi}^p_{,m} + d \sum^n_{j=1} \left(\bar{\phi} \hat{g}_j(z,...,z,t)z^{-1}\right)^p_{,m},
\end{eqnarray*}
where $(\cdot)^p_{,m}$ represents the coefficient of $z^p e^{{\bf i} m t}$.

Set
\begin{eqnarray*}
\Gamma(z,t,h)&=& h - \varepsilon^{-1} n \widehat{f}(z+h,...,z+h,t)- \varepsilon^{-1} \sigma h \\
&& \qquad -d \sigma (n-1) h - d \sum^n_{j=1}h \hat{g}_j(z,...,z,t)z^{-1}.
\end{eqnarray*}
Obviously, $\Gamma(z,t,h)$ is analytic in a neighborhood of $\{0\}\times [0,T]\times \{0\}$, and it satisfies
\[
\Gamma(0,t,0)=0, \quad \frac{\partial \Gamma}{\partial h}(0,t,0)=1 - \varepsilon^{-1} \sigma - d \sigma (n-1), \quad t\in [0,T].
\]
By the linear algebra, we can take an invertible linear transformation such that the elements $\sigma_j$ of the matrix $A$ are suitably  small and keeping the eigenvalues of $A$ do not change. So we have that $\sigma$ is suitably small. Under this surgery we get $\frac{\partial \Gamma}{\partial h}(0,t,0) > 0$, $t\in [0, T]$. By the Implicit Function Theorem, it follows that $\Gamma(z,t,h)=0$ has a unique analytic solution $h(z,t)$ defined in a neighborhood of $\{0\}\times [0,T]$. Then we get from \cite{Hi} or \cite{Zh} that the distinguished normalization
\[
x_s = y_s + \sum_{l \in \mathbb{Z}_+^n ,|l| > 1, m \in \mathbb{Z}} \phi^l_{s, m} y^l e^{{\bf i} m t}, \quad s=1,...,n,
\]
is convergent in  a neighborhood of $\{0\}\times [0,T]$. This proves that system \eqref{F4} is analytically equivalent to its distinguished normal form \eqref{FF1}. Consequently system \eqref{F1} is analytically equivalent to its distinguished normal form \eqref{FF1}.

We complete the proof of the theorem.

\medskip

\noindent{\bf Acknowledgements.} The authors sincerely appreciate the referee for his/her careful reading and the nice comments and suggestions,
which help us to improve the presentation of this paper.

\end{document}